\documentclass[11pt,a4paper]{amsart}
\usepackage[margin=3cm]{geometry}
\usepackage{glossaries}
\usepackage{caption}
\usepackage{graphicx}
\usepackage{amsmath}
\usepackage{amssymb}
\usepackage{physics}
\usepackage{enumerate}
\usepackage[utf8]{inputenc}
\usepackage{amsthm}
\usepackage{mathrsfs}
\usepackage{mathtools}
\usepackage{chngcntr}
\usepackage{xcolor}
\usepackage{bbm}
\usepackage{url}
\usepackage{xfrac}
\usepackage[foot]{amsaddr}
\usepackage[colorlinks]{hyperref}
\usepackage{nicematrix}

\numberwithin{equation}{section}

\allowdisplaybreaks

\newtheorem{thm}{Theorem}[section]
\newtheorem{lem}[thm]{Lemma}
\newtheorem{prop}[thm]{Proposition}
\newtheorem{cor}[thm]{Corollary}

\newtheorem{theorem}{Theorem}

\theoremstyle{remark}
\newtheorem{remark}{Remark}

\theoremstyle{definition}
\newtheorem*{definition}{Definition}

\begin{document}
	
	\title{Finite Riesz products and Ornstein non-inequalities on quantum tori}
	
	\author[C.P. Tantalakis]{Christos P. Tantalakis$^{1,\dagger}$}
	\author[M. Wojciechowski]{Micha\l{} Wojciechowski$^{2,\ddagger}$}
	
	\address{\noindent $^{1}$Faculty of Mathematics, Informatics, and Mechanics\\
		University of Warsaw\\ Warsaw, Poland
		\newline $^{2}$Institute of Mathematics\\ Polish Academy of Sciences\\ Warsaw, Poland}
	
	\email{$^\dagger$c.tantalakis@uw.edu.pl, $^\dagger$chritant@live.com}
	\email{$^{\ddagger}$miwoj@impan.pl}
	
	\thanks{The authors gratefully acknowledge the financial support from NCN (Project 2020/02/Y/ST1/00072)}
	
	\begin{abstract}
		We demonstrate a construction of products on the quantum torus $\mathbb{T}_\theta^2$ that generalises the usual construction of finite Riesz products on the commutative torus $\mathbb{T}^2$. We explain why the former constitutes a natural analogue of the latter in the non-commutative setting and, based on this construction, as well as on previous results by K. Kazaniecki and the second author, we prove a non-commutative version of an Ornstein non-inequality.
	\end{abstract}
	
	
	
	\maketitle
	
	\section{Introduction}\label{Intro}
	
	\subsection{Aim of the article and main techniques}
	
	The article deals with non-commutative analogues of Ornstein's non-inequalities on the the quantum torus. In \cite{KazWojc}, Kazaniecki and the second author used finite Riesz products to construct examples that confirm Ornstein's results. In this article, we show that an analogous finite Riesz product structure provides us with examples of the Ornstein phenomenon for the case of the quantum torus. To the best of the authors' knowledge, Ornstein's non-inequalities and analogues of Riesz products on the non-commutative torus have not yet been considered and one of the main goals of the present article is to initiate such a study.
	
	We recall that Ornstein's celebrated result states that if $\{P(D),Q_1(D),\dots,Q_N(D)\}$ is a collection of homogeneous linear differential operators of the same homogeneity degree which satisfy
	\begin{equation}\label{00}
		\left\| P(D)f \right\|_1\leq C \sum_{j=1}^N \left\| Q_j(D)f \right\|_1, \ \text{for all } f\in C^{\infty}_c(\mathbb{R}^n),
	\end{equation}
	for some positive constant $C$ (that depends on the choice of the differential operators), then the operator $P(D)$ lies in the linear span of $\{Q_j(D)\}_{j=1}^N$ (cf. \cite[Theorem 1]{Ornstein}). As a matter of fact, in \cite{Ornstein}, Ornstein first proves the following simple case in $\mathbb{R}^2$. There does not exist a universal constant $C>0$ such that
	\begin{equation}\label{Ornstein}
		\iint\limits_{[-1,1]^2} \left| \frac{\partial^2f}{\partial x\partial y}(x,y) \right|\dd{x}\dd{y}\leq C \displaystyle{\left( \iint\limits_{[-1,1]^2} \left| \frac{\partial^2f}{\partial x^2}(x,y) \right|\dd{x}\dd{y} + \iint\limits_{[-1,1]^2} \left| \frac{\partial^2f}{\partial y^2}(x,y) \right|\dd{x}\dd{y} \right)},
	\end{equation}
	for any infinitely differentiable function $f$ on the square $[-1,1]^2$ which is annihilated outside this square. We note that (\ref{Ornstein}) is disproved in the first part of \cite{Ornstein}, while the analysis of (\ref{00}) is much deeper and is presented in the second part of his article.
	
	In this article, we define and construct finite Riesz products on the non-commutative torus that violate a non-commutative analogue of (\ref{Ornstein}) (see Theorem \ref{main}). Furthermore, we stress that our results generalise to a wider (compared to the simple 2-dimensional aforementioned case) context according to Theorem \ref{thm0}. Observe that the latter constitutes a non-commutative analogue of the anisotropic case that is discussed in \cite{KazWojc}. Notice that another contribution to Ornstein non-inequalities (for the commutative case) can be also found in \cite{KazStolW}.
	
	The key ingredient to obtain the non-commutative version of Ornstein's non-inequality is an elementary continuity property of the non-commutative $L^1$-norm of trigonometric polynomial sections of the rotation algebras, whose Fourier coefficients depend continuously on the deformation parameter $\theta$ (see Lemma \ref{lem0}). Therefore, the key idea of our proof is to construct small perturbations of $\theta$ in order to get our estimates from the commutative case (i.e. $\theta=0$), that was examined by Kazaniecki and the second author in \cite{KazWojc}.
	
	For completeness of the argument, observe that when $1<p<+\infty$, under certain conditions, an inequality of the form (\ref{00}) (in the commutative case) is possible without implying the linear dependence of the Ornstein case. This can be proved by using, for example, the Mikhlin-H\"{o}rmander multiplier theorem (cf. \cite[Theorem 6.2.7]{Grafakos}). A similar result (when $p\in(1,+\infty)$) for the non-commutative torus is also obtainable. Because a concrete relevant reference remains unknown to the authors, we give a sketch of the approach. In order to tackle the question when $p\in(1,+\infty)$ for the non-commutative torus, the Mikhlin-H\"{o}rmander multiplier theorem yields not only that the involved Fourier multipliers are $L^p$-bounded on the commutative torus $\mathbb{T}^d$, but also completely bounded (a stronger notion than mere boundedness). The latter is a folklore result that can be found in various sources (e.g. \cite{MRX}). The result is finally obtained by applying \cite[Theorem 7.3]{HAoQT}. This important theorem states that the space of completely bounded $L^p$ multipliers on the non-commutative torus is actually the same with the space of completely bounded $L^p$ multipliers on the commutative torus. Regarding the case of $p=+\infty$, the commutative case is treated in \cite{LeeuwMirkil}, where it is proved that there does not exist a collection of homogeneous (of the same homogeneity degree) linearly independent differential operators that satisfy (\ref{00}). In other words, the Ornstein phenomenon occurs in the case of $p=+\infty$ too. The authors believe that the non-commutative analogue of the $p=+\infty$ case can be obtained by reasoning similarly to the $p=1$ case and will be discussed in another article. For completeness, we mention that in the commutative case, when dealing with these kinds of problems, apart from Fourier multipliers, one can use Calderon-Zygmund theory, Sobolev embeddings, etc. These notions possess natural analogues in the setting of quantum tori too. Apart from \cite{HAoQT} and \cite{MRX}, useful information on all of the aforementioned topics can be found also in \cite{JMP}, \cite{Parcet}, \cite{Ricard}, \cite{Spera}, \cite{Suk}, and \cite{XXY}.
	
	\subsection{Structure of the article}
	
	Because the nature of the article lies on different areas of mathematical analysis (e.g. harmonic analysis, von Neumann algebras, operator theory, etc.), Section \ref{Intro} is purely introductory and it aims to give the absolutely necessary background material for this article in a comprehensive way, even if this could be considered to be rather trivial sometimes. The motivation for our work is the concept of Riesz products (defined on the usual commutative torus $\mathbb{T}^2$) as well as the way that they provide the machinery for obtaining an Ornstein non-inequality. The necessary definitions are given in Section \ref{RP&ONI}. Because our main goal is to develop analogues of the Ornstein phenomenon in the non-commutative world, in Section \ref{NC} we briefly discuss the concept of non-commutative integration and Section \ref{IRA} provides us with the definition of the non-commutative (or quantum) torus.
	
	In Section \ref{Prel}, we briefly present the main proving tools that we take from the commutative case. Namely, a particular construction of finite Riesz products and an Ornstein non-inequality on the commutative torus $\mathbb{T}^2$.
	
	Section \ref{NCRP} is devoted to finite non-commutative Riesz products. More precisely, in Section \ref{Def}, we introduce these products and we state our main theorem (i.e. Theorem \ref{main}). Section \ref{Lemmas} comprises the statements and the proofs of all the necessary lemmata that will finally lead to the proof of Theorem \ref{main}, which is demonstrated in Section \ref{Proof}.
	
	Finally, in Section \ref{Concl} we conclude with some instructive remarks, which could possibly motivate further research on the topic, and we also provide some multidimensional extensions of our results. In particular, Theorem \ref{thm0} provides an example of Ornstein's non-inequality for an anisotropic case, in the spirit of the results in \cite{KazWojc}.
	
	\subsection{Framework, main notions and first definitions}
	
	\subsubsection{Riesz products}\label{RP&ONI}
	
	We recall that the term \textit{Riesz product} usually refers to products of the form
	\begin{equation}\label{eq1}
		P_n(t)=\prod_{j=1}^n [1+a_j\cos(2\pi m_j t)] \ \text{and } P_0(t)=1, \ \text{for all } t\in[0,1],
	\end{equation}
	for any arbitrary natural number $n$, where $\{a_j\}_{j\in\mathbb{N}}$ is a sequence of $[-1,1]$ and $\{m_j\}_{j\in\mathbb{N}}$ is a sequence of natural numbers that satisfies the lacunary condition
	\begin{equation}\label{16}
		\frac{m_{j+1}}{m_j}\geq3, \ \text{for all } j\in\mathbb{N}.
	\end{equation}
	Obviously, any such product admits a Fourier series representation
	\begin{equation}\label{eq0}
		P_n(t)=\sum_{j=-N}^N \hat{P}_n(j)e^{2\pi i jt}, \ \text{for all } t\in[0,1],
	\end{equation}
	where $N=\sum_{j=1}^n m_j$. It is interesting to notice that condition (\ref{16}) guarantees that the non-zero Fourier coefficients $\hat{P}_n(j)$ are obtained uniquely in the following sense. For any $j\in\{1,\dots,n\}$, define the sets
	\begin{equation*}
		M_j=\left\{ m_j+\sum_{k=1}^{j-1}\xi_k m_k, \ \text{where } \xi_k\in\{-1,0,1\} \right\}.
	\end{equation*}
	Then
	\begin{equation*}
		P_n(t)=1+\sum_{j=1}^n\sum_{k\in\pm M_j}\hat{P}_n(k)e^{2\pi i kt}, \ \text{for all } t\in[0,1].
	\end{equation*}
	For any $j\in\{1,\dots,n\}$ and any $k\in M_j$, we define $\left\| k \right\|_0$ to denote the unique number of the terms of the sequence $\{m_j\}_{j\in\mathbb{N}}$ that generate $k$; i.e.
	\begin{equation}\label{k0}
		\left\| k \right\|_0=1+\sum_{l=1}^{j-1}\left|\xi_l\right|,
	\end{equation}
	for
	\begin{equation}\label{17}
		k= m_j+\sum_{l=1}^{j-1}\xi_l m_l\in M_j.
	\end{equation}
	Then, for any $j\in\{1,\dots,n\}$ and any $k\in M_j$ of the form (\ref{17}),
	\begin{equation}\label{eq2}
		\hat{P}_n(k)=2^{-\left\| k \right\|_0} a_j \prod_{l\in\Omega_k}a_l,
	\end{equation}
	where
	\begin{equation*}
		\Omega_k=\left\{ l\in\{1,\dots,j-1\}: \ \xi_l\neq0 \right\}.
	\end{equation*}
	
	Riesz products were introduced by F. Riesz in \cite{FRiesz}. For a panorama on this topic, we refer the reader to \cite[Chapter XI]{Bary}, \cite[Chapter 3.6]{Grafakos}, \cite[Chapter 7]{ECHA} and \cite[Chapter V]{Zygmund}. Of course it makes sense to consider multivariable analogues of the above in the following sense. Let d be an arbitrary natural number and consider the sequences $\{a_j\}_{j\in\mathbb{N}}$ of $[-1,1]$ and $\{m_j(k)\}_{j\in\mathbb{N}}$ of $\mathbb{N}$, for $k=1,\dots,d$. Then, for any $n\in\mathbb{N}$, one can define products of the form
	\begin{equation*}
		P_n(t)=\prod_{j=1}^n[1+a_j\cos(2\pi m_j\cdot t)], \ \text{for all } t\in[0,1]^d,
	\end{equation*}
	where, for $m_j=(m_j(1),\dots,m_j(d))$ and $t=(t_1,\dots,t_d)$, $m_j\cdot t$ denotes the usual inner product; i.e.
	\begin{equation*}
		m_j\cdot t=\sum_{k=1}^dm_j(k)t_k.
	\end{equation*}
	Lacunary conditions on the sequence $\{m_j\}_{j\in\mathbb{N}}$ can be imposed coordinate-wise. Products of the latter form appear in \cite{KazWojc}.
	
	\subsubsection{Non-commutative $L^p$-spaces}\label{NC}
	
	Recall that von Neumann algebras can be regarded as (in general) non-commutative analogues of $L^\infty$ spaces. Likewise in the commutative world, where $L^\infty$ is the dual of $L^1$, von Neumann algebras admit a predual which is considered as a non-commutative analogue of $L^1$ (cf. \cite[Theorem III.3.5]{TakI}). For a complete panorama on C$^*$ and von Neumann algebras we refer the reader to \cite{Conway}, \cite{CAbE}, \cite{KadRingI}, \cite{Pedersen} or \cite{TakI}.
	
	Under certain conditions, it is possible to proceed one step further and establish a whole non-commutative $L^p$-theory. In fact, if $\mathcal{A}$ is a von Neumann algebra with a faithful, semi-finite, normal trace $\tau$, then the (in general) \textit{non-commutative $L^p(\mathcal{A},\tau)$-space}, for $p\in(0,+\infty)$, is defined to be the space of all operators $T$ that are affiliated with $\mathcal{A}$ and
	\begin{equation*}
		\left\| T \right\|_p:=\tau\left( \left|T\right|^p \right)^{\sfrac{1}{p}}<+\infty,
	\end{equation*}
	where $\left|T\right|=\sqrt{T^*T}$. For $p=+\infty$, the \textit{non-commutative $L^\infty(\mathcal{A},\tau)$-space} comprises all the operators $T$ that are affiliated with $\mathcal{A}$ and
	\begin{equation*}
		\left\| T \right\|_\infty:=\left\| T \right\|<+\infty.
	\end{equation*}
	In other words, $L^\infty(\mathcal{A},\tau)$ is the space of bounded operators, affiliated with $\mathcal{A}$. Bear in mind that whenever the von Neumann algebra, as well as its faithful, semi-finite, normal trace $\tau$ are implicit and there is no danger of confusion, the notation $L^p(\mathcal{A},\tau)$ will be reduced for simplicity to $L^p$. Notice that, as in the commutative case, for $p\geq1$, the defined spaces are Banach (cf. \cite[Section 1.1]{FackKosaki}). As expected, those spaces are quasi-Banach for $0<p<1$, a result due to K.-S. Saito \cite{Saito}.
	
	\begin{remark}
		Notice that when an operator is bounded, affiliation is equivalent to belonging to the algebra. Because our framework will be that of bounded operators, there is no need to define the notion of affiliation. For these, as well as for more information about affiliated operators, the reader is advised to check \cite[Chapter 9]{Lectures}.
	\end{remark}
	
	Non-commutative integration in this setting was initiated by J. Dixmier \cite{Dixmier} and I. E. Segal \cite{SegalII}. Since then, those spaces have attracted significant interest with much effort to be put on examining if the basic theory and properties of the usual commutative $L^p$-spaces can be translated in the non-commutative vocabulary. For an extensive, a bit more comprehensive, panorama on non-commutative integration, we refer to \cite{FackKosaki}, \cite{Nelson}, \cite{Saito}, \cite{TakII}, and \cite{Yeadon}. Finally, for a comprehensive exposition of notions relevant to weights (or traces) and affiliated operators, apart from the aforementioned sources, we refer the reader in particular to \cite[Chapter 7.5]{KadRingII} and \cite[Chapter 5.6]{KadRingI}, respectively.
	
	\subsubsection{Rotation algebras and non-commutative/quantum tori}\label{IRA}
	
	In this section, we present the definitions of the main objects of interest in this article; i.e. rotation algebras as well as the non-commutative/quantum tori. Regarding the latter, we highlight the fact that in general do not appear in the literature as identical objects. Nonetheless, the distinction between the two terms is rather niche and does not serve our purposes so, in the context of our work, the two terms will be considered identical and are used interchangeably. In fact, this is not something unusual, since there are examples in the literature where the two terms define exactly the same object (cf. \cite{Suk}).
	
	Let $\theta$ be a number in $[0,1]$ and denote by $\mathcal{A}_\theta$ the universal C$^*$-algebra generated by two unitary elements $U_\theta$ and $V_\theta$ that satisfy the ``commutation" relation below:
	\begin{equation}\label{rotation}
		U_\theta V_\theta=e^{2\pi i\theta}V_\theta U_\theta.
	\end{equation}
	Namely, if $\mathcal{P}_\theta$ denotes the $*$-algebra that contains the non-commutative polynomials of the form
	\begin{equation}\label{18}
		p=p(U_\theta,V_\theta)=\sum_{\alpha=(\alpha_1,\alpha_2)\in\mathbf{N}}p_\alpha U_\theta^{\alpha_1}V_\theta^{\alpha_2},
	\end{equation}
	where $\mathbf{N}$ is any finite subset of $\mathbb{Z}^2$, then
	\begin{equation*}
		\mathcal{A}_\theta=\overline{\mathcal{P}_\theta},
	\end{equation*}
	where the closure is considered with respect to the norm of $\mathcal{P}_\theta$. Then $\mathcal{A}_\theta$ is a \textit{rotation algebra}. If $\theta$ is rational, then $\mathcal{A}_\theta$ is a \textit{rational rotation algebra}. If otherwise, $\mathcal{A}_\theta$ is an \textit{irrational rotation algebra}. In addition, we define the functional $\tau_\theta:\mathcal{P}_\theta\to\mathbb{R}$ by
	\begin{equation}\label{trace1}
		\tau_\theta(p)=p_{\mathbf{0}},
	\end{equation}
	where $\mathbf{0}:=(0,0)$, which can be extended to a faithful tracial state on $\mathcal{A}_\theta$. The coefficients $p_\alpha$ in (\ref{18}) are the \textit{Fourier coefficients} of $p$ (cf. \cite{HAoQT} or \cite{RieffelNCtori}) and
	\begin{equation*}
		p_\alpha=\hat{p}(\alpha)=\tau_\theta\left( (U_\theta^{\alpha_1}V_\theta^{\alpha_2})^*p \right), \ \text{for any } \alpha\in\mathbf{N}.
	\end{equation*}
	
	We recall an analogue of the derivative in rotation algebras; i.e the derivations. In general, a \textit{derivation} of a C$^*$-algebra is a linear operator $\delta:\mathcal{A}\to\mathcal{A}$ such that
	\begin{equation*}
		\delta(ab)=\delta(a)b+a\delta(b), \ \text{for all } a,b\in\mathcal{A};
	\end{equation*}
	see for example, \cite[Chapter 8.6]{Pedersen}. In rotation algebras, we define the following kind of derivations (cf. \cite{HAoQT}), which give an analogue of differentiation in the commutative case. For $j=1,2$,
	\begin{equation}\label{6}
		\delta_j(U_\theta^{\alpha_1}V_\theta^{\alpha_2})=2\pi i\alpha_jU_\theta^{\alpha_1}V_\theta^{\alpha_2}, \ \text{for any } \alpha=(\alpha_1,\alpha_2)\in\mathbb{Z}^2,
	\end{equation}
	and we extend $\delta_j$ linearly on $\mathcal{P}_\theta$. Notice that $\delta_1$ and $\delta_2$ commute. We denote their product, $\delta_1\delta_2$, by $\delta_{12}$.
	
	We define the \textit{non-commutative/quantum torus} $\mathbb{T}_\theta^2$ as the rotation algebra $\mathcal{A}_\theta$. We also make the convention that the angle of rotation $\theta$ is irrational, which is also something common in the literature (cf. \cite{Spera} or \cite{Suk}). In fact, while in \cite{Spera} is not explicitly stated that $\theta$ is irrational, this is implicitly assumed and it can be inferred from the claim that $\mathcal{A}_\theta$ has a unique faithful tracial state (cf. \cite[p. 145]{Spera}). As a matter of fact, this is true for any irrational rotation algebra (cf. \cite[Chapter VI]{CAbE}). In addition, for fixed $\theta$, any C$^*$-algebra generated by unitaries $U$ and $V$ that satisfy (\ref{rotation}) is canonically isomorphic to $\mathbb{T}_\theta^2$ (cf. \cite[Chapter VI]{CAbE}).
	
	For orientation, one may realise $\mathbb{T}_\theta^2$ as the C$^*$-subalgebra of the algebra of bounded operators on $L^2(\mathbb{T}^2)$, generated by the unitary operators $U_\theta$ and $V_\theta$ which are described by
	\begin{equation}\label{UV}
		U_\theta:e_{m,n}\mapsto e_{m+1,n} \ \ \text{and } \ V_\theta:e_{m,n}\mapsto e^{-2\pi im\theta}e_{m,n+1},
	\end{equation}
	where $\{e_{m,n}\}_{m,n\in\mathbb{Z}}$ is the usual orthonormal basis. Notice that $U_\theta$ and $V_\theta$ satisfy (\ref{rotation}) and that, in this representation, the canonical trace defined in (\ref{trace1}) is given by
	\begin{equation*}
		\tau_\theta(T)=\left(Te_{\mathbf{0}},e_{\mathbf{0}}\right), \ \text{for any } T\in\mathbb{T}_\theta^2,
	\end{equation*}
	where $\left(\cdot,\cdot\right)$ is the usual inner product on $L^2(\mathbb{T}^2)$. Nonetheless, in what follows we do not need to make use of any particular representation of $\mathbb{T}_\theta^2$.
	
	We close our introduction by mentioning some examples of the most historic and important contributions to the topic. These include \cite{Connes}, \cite{PimVoic1} and \cite{RieffelCstar}. Moreover, further information on the topic is available in \cite{HAoQT}, \cite{ConMos}, \cite{CAbE}, \cite{RieffelNCtori} and \cite{Weaver}.
	
	\section{Preliminaries}\label{Prel}
	
	In this small section, we present the two most important tools for the proof of our main result. First, in \cite{KazWojc}, Kazaniecki and the second author, based on finite Riesz products, construct examples of Ornstein non-inequalities. More precisely, they work with Riesz products of the following form. For a sequence of natural numbers $\{m_j\}_{j\in\mathbb{N}}$, let
	\begin{equation}\label{Mj'}
		M_j:=\left\{ k\in\mathbb{Z}^2: \ k=((-1)^{j-1}m_j,m_j)+\sum_{l=1}^{j-1}\xi_l((-1)^{l-1}m_l,m_l), \ \xi_l\in\{0,\pm1\} \right\}.
	\end{equation}
	Then define the Riesz products
	\begin{equation}\label{2_0}
		P_n(s,t)=1+\sum_{j=1}^n\sum_{k=(k_1,k_2)\in\pm M_j}2^{-\left\| k \right\|_0}e^{2\pi i (k_1s+k_2t)}, \ \text{for all } (s,t)\in[0,1]^2,
	\end{equation}
	where $\left\| \cdot \right\|_0$ being as in (\ref{k0}). Notice that this is obtained after expanding the products (compare for example with the formulae for the 1-dimensional case, (\ref{eq1}), (\ref{eq0}) and (\ref{eq2})). Based on these products, they build the trigonometric polynomials $W_n$, which they call modified Riesz products and result the Ornstein's non-inequality as follows. For any $n\in\mathbb{N}$, let
	\begin{equation}\label{0}
		W_n(s,t):=-\frac{1}{4\pi^2}\sum_{j=1}^n\sum_{k\in\pm M_j}2^{-\left\| k \right\|_0}k_2^{-2}e^{2\pi i (k_1s+k_2t)}, \ \text{for all } (s,t)\in[0,1]^2.
	\end{equation}
	Notice that the latter are closely related with the original Riesz products, since
	\begin{equation*}
		\frac{\partial^2}{\partial t^2}W_n(s,t)=P_n(s,t)-1
	\end{equation*}
	and
	\begin{equation*}
		\frac{\partial^2}{\partial s^2}W_n(s,t)=B_{n}(s,t)+(P_n(s,t)-1), \ \text{for } (s,t)\in[0,1]^2,
	\end{equation*}
	where $B_{n}$ is a trigonometric polynomial with Fourier coefficients
	\begin{equation*}\label{Bn}
		\hat{B}_{n}(k)=\left[ \left( \frac{k_1}{k_2} \right)^2-1 \right]\hat{P}_n(k), \ \text{for all } k=(k_1,k_2)\in M_j \ \text{and } j=1,\dots,n.
	\end{equation*}
	Moreover,
	\begin{equation*}
		\frac{\partial^2}{\partial s\partial t}W_n(s,t)=E_{n}(s,t)+G_n(s,t), \ \text{for } (s,t)\in[0,1]^2,
	\end{equation*}
	where $E_{n}$ and $G_n$ are trigonometric polynomials with the following Fourier coefficients:
	\begin{equation*}\label{EG}
		\hat{E}_{n}(k)=\left( \frac{k_1}{k_2}-(-1)^{j-1} \right)\hat{P}_n(k) \ \ \text{and } \ \hat{G}_n(k)=(-1)^{j-1}\hat{P}_n(k),
	\end{equation*}
	for all $k=(k_1,k_2)\in M_j$ and $j=1,\dots,n$. The latter implies that
	\begin{equation*}
		G_n(s,t)=\sum_{j=1}^n(-1)^j\left[ P_j(s,t)-P_{j-1}(s,t) \right], \ \text{for all } (s,t)\in[0,1].
	\end{equation*}
	Then it is proved that the $L^1$-norms of the trigonometric polynomials $B_n$ and $E_n$ are uniformly bounded. On the other hand, the alternating sign in the expression of $G_n$ and a fast enough growth of the frequencies $m_j$ lead to a blowing-up $L^1$-norm of $G_n$, as $n\to+\infty$. As a matter of fact, this can be summarised to the following theorem (cf. \cite[Theorem 2.3]{KazWojc}).
	
	\begin{theorem}\label{Com}
		
		There exist constants $K_1>0$ and $C_1>0$ such that, for any natural number $n$ the modified finite Riesz products $W_n$ (defined in (\ref{0})), with frequencies $\{m_j\}_{j\in\mathbb{N}}$ which satisfy the following growth condition
		\begin{equation}\label{growth}
			\frac{m_j}{m_{j+1}}\leq 3^{-2n}, \ \text{for all } j\in\mathbb{N},
		\end{equation}
		satisfy the following:
		\begin{equation*}\label{com}
			\max\left\{\left\|\frac{\partial^2}{\partial s^2}W_n\right\|_{L^1(\mathbb{T}^2)}, \left\|\frac{\partial^2}{\partial t^2}W_n\right\|_{L^1(\mathbb{T}^2)}\right\}\leq K_1 \ \ \text{and } \ \left\|\frac{\partial^2}{\partial s\partial t}W_n\right\|_{L^1(\mathbb{T}^2)}\geq C_1n.
		\end{equation*}
	\end{theorem}
	
	\begin{remark}\label{rem2}
	It can be checked that the degree of the trigonometric polynomials $W_n$ increases like $e^{n^2}$, as $n\to+\infty$.
	\end{remark}
	
	\section{Non-commutative Riesz products}\label{NCRP}
	
	\subsection{Definition of the products and the main result}\label{Def}
	
	\noindent Going back to the products of the form (\ref{2_0}), we can rewrite them as
	\begin{equation*}
		P_n(s,t)=\prod_{j=1}^n\left[ 1+\frac{1}{2}\left( e^{2\pi i m_j((-1)^{j-1}s+t)} + e^{-2\pi i m_j((-1)^{j-1}s+t)} \right) \right], \ \text{for all } (s,t)\in[0,1]^2.
	\end{equation*}
	Motivated by (\ref{UV}), the above can be thought of as the value of the polynomial below for $\theta=0$;
	\begin{equation*}
		P_{\theta,n}=P_{\theta,n}(U_\theta,V_\theta)=\prod_{j=1}^{n}\left[ I +\frac{1}{2}\left( U_\theta^{(-1)^{j-1}m_j}V_\theta^{m_j}+V_\theta^{-m_j}U_\theta^{(-1)^jm_j} \right) \right],
	\end{equation*}
	for $\theta\in[0,1]$. This defines a sequence of (in general) non-commutative Riesz products. Because of commutation limitations, in order the above product to be meaningful, we define the product in the following sense.
	\begin{definition}
		Let $\mathcal{A}_\theta$ be a rotation algebra with generators $U_\theta$ and $V_\theta$. For a sequence of natural numbers $\{m_j\}_{j\in\mathbb{N}}$ that satisfies the lacunary condition
		\begin{equation*}
			\frac{m_{j+1}}{m_j}\geq 3, \ \text{for all } j\in\mathbb{N},
		\end{equation*}
		we define the factors
		\begin{equation}\label{2_1}
			p_{\theta,j}:= I+\frac{1}{2}\left( U_\theta^{(-1)^{j-1}m_j}V_\theta^{m_j}+V_\theta^{-m_j}U_\theta^{(-1)^jm_j} \right), \ \text{for all } j\in\mathbb{N},
		\end{equation}
		and the (non-commutative) Riesz products
		\begin{equation}\label{2_4}
			P_{\theta,n}:=p_{\theta,n}p_{\theta,n-1}\dots p_{\theta,1}, \ \text{for all } n\in\mathbb{N}.
		\end{equation}
	\end{definition}
	Notice that $p_{\theta,j}\geq0$ so, formula (\ref{2_4}) constitutes a natural analogue of finite Riesz products on the non-commutative torus. In fact, when restricted to $\theta=0$, we are reduced to the usual commutative Riesz products. Nonetheless, there are a couple of caveats the reader should be aware of and which are discussed in the following remark.
	
	\begin{remark}\label{rem0}
		First, it is important to notice that in general $P_{\theta,n}$ is not non-negative; for otherwise, this would demand the factors $p_{\theta,j}$ to be pairwise commuting. Moreover, notice that the Fourier coefficients of $P_{\theta,n}$ are not always identical with those of $P_n$ (see (\ref{2_0})). In general, what is true is that they are equal to those of $P_n$ scaled by a unimodular complex number that depends on $\theta$ (because of the rotation relation (\ref{rotation})). This means that while it would be tempting to define the non-commutative Riesz products as the trigonometric polynomials of $U_\theta$ and $V_\theta$ that have the same Fourier coefficients with the respective commutative ones, this would risk calling products elements of $\mathcal{A}_\theta$ which may be far from being products of the desired factors.
	\end{remark}

	By restricting ourselves back to the Riesz products of the form described by (\ref{2_1}) and (\ref{2_4}), we are now ready to state our main result, which generalises Theorem \ref{Com}. But first we also need to define the non-commutative analogues of the commutative trigonometric polynomials $W_n$ in (\ref{0}). So, for any natural number $n$ and any $\theta\in[0,1]$, let
	\begin{equation}\label{Wn}
		W_{\theta,n}:=-\frac{1}{4\pi^2}\sum_{j=1}^n\sum_{k\in M_j}k_2^{-2}\left( \hat{P}_{\theta,n}(k)U_\theta^{k_1}V_\theta^{k_2}+\hat{P}_{\theta,n}(-k)U_\theta^{-k_1}V_\theta^{-k_2} \right),
	\end{equation}
	where the sets $M_j$ are given by (\ref{Mj'}). We draw the reader's attention to fact that, in contrast with the commutative case, $\hat{P}_{\theta,n}(k)\neq\hat{P}_{\theta,n}(-k)$, but
	\begin{equation*}\label{FP}
		\left| \hat{P}_{\theta,n}(\pm k) \right|=\left| \hat{P}_{n}(\pm k) \right|=2^{-\left\| k \right\|_0},
	\end{equation*}
	where $\left\| \cdot \right\|_0$ is defined in (\ref{k0}); see also Remark \ref{rem0}. As in the commutative case, notice that
	\begin{equation}\label{3}
		\delta^2_2(W_{\theta,n})=P_{\theta,n}-1, \ \delta^2_1(W_{\theta,n})=B_{\theta,n}+(P_{\theta,n}-1), \ \text{for all } n\in\mathbb{N},
	\end{equation}
	and
	\begin{equation}\label{5}
		\delta_{12}(W_{\theta,n})=E_{\theta,n}+G_{\theta,n}, \ \text{for all } n\in\mathbb{N},
	\end{equation}
	where, for any $j=1,\dots,n$ and any $k=(k_1,k_2)\in \pm M_j$,
	\begin{equation}\label{4}
		\hat{B}_{\theta,n}(k)=\left[\left(\frac{k_1}{k_2}\right)^2-1\right]\hat{P}_{\theta,n}(k), \ \hat{E}_{\theta,n}(k)=\left( \frac{k_1}{k_2}-(-1)^{j-1} \right)\hat{P}_{\theta,n}(k)
	\end{equation}
	and
	\begin{equation}\label{Gn}
		\hat{G}_{\theta,n}(k)=(-1)^{j-1}\hat{P}_{\theta,n}(k).
	\end{equation}
	
	\begin{thm}\label{main}
		There exist positive constants $K$ and $C$, such that for any irrational $\theta\in[0,1]$ and for any natural number $N$, there exists a trigonometric polynomial $Q$ of $\mathbb{T}_\theta^2$, obtained from a finite non-commutative Riesz product, that depends on $N$, such that
		\begin{equation}\label{2_21}
			\max_{j=1,2}\left\{\left\|\delta^2_jQ\right\|_{L^1(\mathbb{T}_\theta^2)}\right\}\leq K \ \ \text{and } \ \left\|\delta_{12}Q\right\|_{L^1(\mathbb{T}_\theta^2)}\geq C N,
		\end{equation}
		where the derivations $\delta$ are defined in (\ref{6}).
	\end{thm}
	
	\begin{remark}\label{rem1}
	Notice that for $\theta=0$, $\mathbb{T}_0^2$ (i.e. the universal C$^*$-algebra generated by two commuting unitaries) is isomorphic to $C(\mathbb{T}^2)$. In fact, there exists a $\ast$-isomorphism $\phi:\mathbb{T}_0^2\to C(\mathbb{T}^2)$ such that
	\begin{equation*}
		\phi(U_0)(s,t)=e^{2\pi is} \ \ \text{and } \ \phi(V_0)(s,t)=e^{2\pi i t}, \ \text{for } (s,t)\in[0,1]^2.
	\end{equation*}	
	Then, by a density argument, the canonical trace $\tau_0$ on $\mathbb{T}_0^2$ and the integration $\tau$ on $\mathbb{T}^2$ with respect to the normalised Haar measure are linked with the following relation
	\begin{equation*}
		\tau(\phi(a))=\tau_0(a), \ \text{for all } a\in \mathbb{T}^2_0.
	\end{equation*}
	Moreover, since $\phi$ is a $\ast$-homomorphism, continuous functional calculus yields
	\begin{equation*}
		\phi(\left|a\right|)=\left|\phi(a)\right|, \ \text{for any } a\in\mathbb{T}_0^2.
	\end{equation*}
	Thus, $\left\| a \right\|_{L^1(\mathbb{T}_0^2)}=\left\| \phi(a) \right\|_{L^1(\mathbb{T}^2)}$, for all $a$ in $\mathbb{T}_0^2$.
	\end{remark}
	
	\subsection{Preliminary lemmata}\label{Lemmas}
	
	In this section, we present some necessary lemmata that will eventually allow us to link the non-commutative $L^1$-norms of Riesz products with the commutative case and make use of Theorem \ref{Com}. We note that many of the next lemmata may be rather standard, but we include the proofs for sake of completeness and for the reader's convenience.
	
	\begin{lem}\label{lem0}
	Let $N$ be a natural number and
	\begin{equation*}
		P(\theta)=\sum_{\left|m\right|,\left|n\right|\leq N} p_{m,n}(\theta)U_{\theta}^mV_{\theta}^n
	\end{equation*}
	be a trigonometric polynomial in $\mathbb{T}_\theta^2$ such that its Fourier coefficients $p_{m,n}(\theta)$ are continuous functions of $\theta$ on $[0,1]$. Then the map $\theta\mapsto\left\| P(\theta) \right\|_{L^1(\mathbb{T}_\theta^2)}$ is continuous on $[0,1]$.
	\end{lem}
	
	\begin{proof}
	Let $\theta_0\in[0,1]$ and $\varepsilon>0$. For
	\begin{equation*}
		M=\sum_{\left|m\right|,\left|n\right|\leq N}\left\|p_{m,n}(\theta)\right\|_{L^\infty([0,1])}
	\end{equation*}
	notice that $\left\|P(\theta)\right\|\leq M$, for all $\theta\in[0,1]$, where $M$ is independent of $\theta$. Moreover, for $f(t)=\sqrt{t}$, there exists a real non-negative polynomial $Q$ on $[0,M^2]$ such that
	\begin{equation*}
		\left\| f-Q \right\|_{L^\infty([0,M^2])}<\frac{\varepsilon}{4}.
	\end{equation*}
	By applying continuous functional calculus for $P^*(\theta)P(\theta)$, we get that, for all $\theta\in[0,1]$,
	\begin{equation}\label{0'}
		\left\| f\left(P^*(\theta)P(\theta)\right)-Q\left(P^*(\theta)P(\theta)\right) \right\|= \left\| \left|P(\theta)\right| - Q\left( P^*(\theta)P(\theta) \right) \right\|<\frac{\varepsilon}{4}.
	\end{equation}
	Moreover,
	\begin{equation*}
		\left| \left\| P(\theta) \right\|_{L^1(\mathbb{T}_\theta^2)} -\left\| P(\theta_0) \right\|_{L^1(\mathbb{T}_{\theta_0}^2)} \right| = \left| \tau_\theta(\left|P(\theta)\right|) - \tau_{\theta_0}(\left|P(\theta_0)\right|) \right|
	\end{equation*}
	and the right hand side is bounded by
	\begin{multline}\label{1'}
		\big| \tau_\theta\big(\left|P(\theta)\right|-\left|Q\left( P^*(\theta)P(\theta) \right)\right|\big)\big| +\\ + \left| \tau_\theta\big(\left|Q(P^*(\theta)P(\theta))\right|\big) - \tau_{\theta_0}\big(\left|Q(P^*(\theta_0)P(\theta_0))\right|\big) \right|
		+ \big| \tau_{\theta_0}\big(\left|P(\theta_0)\right|-\left|Q\left( P^*(\theta_0)P(\theta_0) \right)\right|\big)\big|.
	\end{multline}
	Notice that, because of the positivity of $Q$, $\left|Q(P^*P)\right|=Q(P^*P)$ and then, by applying (\ref{0'}), the first and the third summands in the above relation are bounded by $\frac{\varepsilon}{4}$. In order to bound the middle term, notice that the positivity of $Q$ together with the continuity of the Fourier coefficients $p_{m,n}(\theta)$ and the commutation relation (\ref{rotation}) imply that, for any $\theta\in[0,1]$, $Q(P^*(\theta)P(\theta))$ is a non-negative trigonometric polynomial of $\mathbb{T}_\theta^2$ of the following form
	\begin{equation*}
		Q(P^*(\theta)P(\theta))=\sum_{\left|m\right|,\left|n\right|\leq K} c_{m,n}(\theta)U^m_\theta V^n_\theta, \ \text{for some } K\in\mathbb{N},
	\end{equation*}
	where the Fourier coefficients $c_{m,n}(\theta)$ are continuous on $[0,1]$. Therefore,
	\begin{equation*}
		\tau_\theta\big(\left|Q(P^*(\theta)P(\theta))\right|\big)= c_{0,0}(\theta), \ \text{for all } \theta\in[0,1],
	\end{equation*}
	and thus, the middle term in (\ref{1'}) equals $\left|c_{0,0}(\theta)-c_{0,0}(\theta_0)\right|$. By continuity, there exists a neighbourhood $\mathcal{U}$ of $\theta_0$ such that for any $\theta\in\mathcal{U}$,
	\begin{equation*}
		\left|c_{0,0}(\theta)-c_{0,0}(\theta_0)\right|<\frac{\varepsilon}{2}.
	\end{equation*}
	We have already seen that the first and third summands in (\ref{1'}) are bounded by $\frac{\varepsilon}{4}$. Therefore, by putting all the estimates together, we obtain that
	\begin{equation*}
		\left| \left\| P(\theta) \right\|_{L^1(\mathbb{T}_\theta^2)} -\left\| P(\theta_0) \right\|_{L^1(\mathbb{T}_{\theta_0}^2)} \right| <\varepsilon, \ \text{for all } \theta\in\mathcal{U}.
	\end{equation*}
	\end{proof}
	
	\begin{remark}
	In what follows, we will often make use of negative deformation parameters $\theta$, despite the fact that $\theta$ is generally considered in $[0,1]$. This is not an abuse of notation, since the commutation relation (\ref{rotation}) presents periodicity. More precisely, the algebras $\mathcal{A}_\theta$ and $\mathcal{A}_{\theta'}$ will be considered identical for any real $\theta$ and $\theta'$ such that $\theta\equiv\theta'\bmod{\mathbb{Z}}$.
	\end{remark}
	
	\begin{lem}\label{lem1}
		Let $W_{\theta,n}$ be the polynomials defined in (\ref{Wn}). Then, for any $\varepsilon>0$ and any natural number $n$, there exists some $\theta_0=\theta_0(n,\varepsilon)$ in $(0,1)$ such that
		\begin{equation*}
			\left| \left\| P_{\theta,n}-1 \right\|_{L^1(\mathbb{T}_\theta^2)} - \left\| P_n-1 \right\|_{L^1(\mathbb{T}^2)} \right|<\varepsilon, \ \text{for all } \theta\in[-\theta_0,\theta_0].
		\end{equation*}
		Therefore, by (\ref{3}),
		\begin{equation*}
			\left| \left\| \delta^2_2 W_{\theta,n} \right\|_{L^1(\mathbb{T}_\theta^2)} - \left\| \delta_2^2 W_n \right\|_{L^1(\mathbb{T}^2)} \right|<\varepsilon, \ \text{for all } \theta\in[-\theta_0,\theta_0].
		\end{equation*}
	\end{lem}
	
	\begin{proof}
	Notice that, because of the commutation relation (\ref{rotation}), the Fourier coefficients of the Riesz products $P_{\theta,n}$ are continuous functions of $\theta$. Then the result is obtained by an immediate application of Lemma \ref{lem0} and Remark \ref{rem1}.
	\end{proof}
	
	By noticing that $\delta^2_1(W_{\theta,n})$ and $\delta_{12}(W_{\theta,n})$ are trigonometric polynomials of $\mathbb{T}_\theta^2$ with Fourier coefficients that depend continuously on $\theta$ (see (\ref{3})-(\ref{Gn})) and arguing similarly to the previous Lemma, we obtain the following.
	
	\begin{lem}\label{lem3}
		Let $W_{\theta,n}$ be the polynomials defined in (\ref{Wn}). Then, for any $\varepsilon>0$ and any natural number $n$, there exists some $\theta_0=\theta_0(n,\varepsilon)$ in $(0,1)$ such that
		\begin{equation}\label{2}
			\max_{i,j=1,2}\left\{\left|\left\|\delta_i\delta_j(W_{\theta,n})\right\|_{L^1(\mathbb{T}_\theta^2)} - \left\|\delta_i\delta_j(W_{n})\right\|_{L^1(\mathbb{T}^2)}\right|\right\}<\varepsilon, \ \text{for any } \theta\in[-\theta_0,\theta_0].
		\end{equation}
	\end{lem}
	
	\subsection{Proof of Theorem \ref{main}}\label{Proof}
	
	The proof is actually based on two steps. The first one is that for small values of rotation $\theta$, the non-commutative Riesz products that are defined in (\ref{2_4}) behave similarly to their commutative counterparts. Thus, the result is obtained by Theorem \ref{Com}. The second step makes use of irrationality. More precisely, irrationality guarantees that the non-commutative torus is rich enough so that, for a certain pattern of high frequencies, we can retrieve a kind of ``almost" commutativity. This helps us to reduce our analysis to the first step.
	
	\begin{proof}
		Let $K_1$ and $C_1$ be the constants from Theorem \ref{Com} and define
		\begin{equation*}
			K=K_1+\frac{1}{2} \ \ \text{and } \ C=\frac{C_1}{2}.
		\end{equation*}
		Now let $N$ be an arbitrary natural number and $\{m_j\}_{j\in\mathbb{N}}$ be a sequence of natural numbers that satisfies the growth condition (\ref{growth}). Then Lemma \ref{lem3} implies that for
		\begin{equation*}
			0<\varepsilon<\min\{K_1,\tfrac{1}{2},C\},
		\end{equation*}
		there exists some $\theta_0=\theta_0(N,\varepsilon)\in(0,1)$ such that (\ref{2}) is satisfied. But, by Theorem \ref{Com},
		\begin{equation*}
			\max_{j=1,2}\left\{ \left\| \delta_j^2(W_N) \right\|_{L^1(\mathbb{T}^2)} \right\}\leq K_1 \ \ \text{and } \ \left\| \delta_{12}W_N \right\|_{L^1(\mathbb{T}^2)}\geq C_1N.
		\end{equation*}
		Therefore, (\ref{2_21}) is valid with $Q=W_{\theta,N}$, where the frequencies $\{m_j\}_{j\in\mathbb{N}}$ of the polynomials $W_{\theta,N}$ (as the latter defined in (\ref{Wn})) satisfy the growth condition (\ref{growth}).
		
		It remains to address the case for $\left|\theta\right|\geq\theta_0$. If $\theta$ belongs outside the interval $(-\theta_0,\theta_0)$ and it is irrational, then there exists some natural number $M=M(\theta,N)$, such that
		\begin{equation*}
			\tilde{\theta}\equiv M^2\theta\bmod\mathbb{Z},
		\end{equation*}
		for some irrational $\tilde{\theta}\in(0,\theta_0)$. This is a consequence of Weyl's equidistribution theorem; for a discussion about this topic, see \cite{KuiNie}. Then, for
		\begin{equation*}
			\tilde{U}:=U_\theta^{M} \ \ \text{and } \tilde{V}:=V_\theta^{M},
		\end{equation*}
		the commutation relation (\ref{rotation}) gives that $\tilde{U}$ and $\tilde{V}$ are unitary operators that satisfy
		\begin{equation*}
			\tilde{U}\tilde{V}=e^{2\pi i\tilde{\theta}}\tilde{V}\tilde{U}.
		\end{equation*}
		We define the C$^*$-algebra $\tilde{\mathcal{A}}$ (a sub-algebra of $\mathcal{A}_\theta$) generated by $\tilde{U}$ and $\tilde{V}$. Then the mapping $\rho:(U_{\tilde{\theta}},V_{\tilde{\theta}})\mapsto(\tilde{U},\tilde{V})$ defines a C$^*$-isomorphism between the irrational rotation algebra $\mathbb{T}^2_{\tilde{\theta}}$ and $\tilde{\mathcal{A}}$. Furthermore, the linear functional $\ell:=\tau_\theta\circ\rho$ (where $\tau_\theta$ is the faithful tracial state on $\mathcal{A}_\theta$) defines a faithful tracial state on $\mathcal{A}_{\tilde{\theta}}$. But we have seen (cf. Section \ref{IRA}) that $\mathcal{A}_{\tilde{\theta}}$ has a unique such state; i.e. $\tau_{\tilde{\theta}}$. Therefore, $\ell=\tau_{\tilde{\theta}}$ and thus,
		\begin{equation}\label{1}
			\left\| a \right\|_{L^1(\mathbb{T}_\theta^2)}=\left\| \rho^{-1}(a) \right\|_{L^1(\mathbb{T}_{\tilde{\theta}}^2)}, \ \text{for any } a\in\tilde{\mathcal{A}}.
		\end{equation}
		Choose a sequence of natural numbers $\{m_j\}_{j\in\mathbb{N}}$ that satisfies the growth condition (\ref{growth}) and let
		\begin{equation*}
			\tilde{P}_{\theta,n}:=\left[ I +\frac{1}{2}\left( {\tilde{U}}^{(-1)^{n-1}m_n}{\tilde{V}}^{m_n}+{\tilde{V}}^{-m_n}{\tilde{U}}^{(-1)^n m_n} \right) \right] \tilde{P}_{\theta,n-1}, \ \text{for all } n\geq2,
		\end{equation*}
		where
		\begin{equation*}
			\tilde{P}_{\theta,1}:= I +\frac{1}{2}\left( {\tilde{U}}^{m_1}{\tilde{V}}^{m_1}+{\tilde{V}}^{-m_1}{\tilde{U}}^{-m_1} \right).
		\end{equation*}
		In addition,
		\begin{equation*}
			\tilde{W}_{\theta,n}:=-\left(\frac{1}{2\pi M}\right)^2\sum_{j=1}^n\sum_{k\in M_j}k_2^{-2}\left( \hat{\tilde{P}}_{\theta,n}(k)\tilde{U}^{k_1}\tilde{V}^{k_2}+\hat{\tilde{P}}_{\theta,n}(-k)\tilde{U}^{-k_1}\tilde{V}^{-k_2} \right).
		\end{equation*}
		Then notice that
		\begin{equation*}
			\rho(W_{\tilde{\theta},n})=M^2\tilde{W}_{\theta,n}, \ \text{for any } n\in\mathbb{N},
		\end{equation*}
		where $W_{\tilde{\theta},n}$ are defined in (\ref{Wn}). If $\tilde{\delta}_j$, for $j=1,2$, are the derivations of $\mathbb{T}_{\tilde{\theta}}^2$, then for any natural number $n$,
		\begin{equation*}
			\tilde{\delta}_i\tilde{\delta}_jW_{\tilde{\theta},n}=(\rho^{-1}\circ\delta_i\delta_j)(\tilde{W}_{\theta,n}), \ \text{for } i,j=1,2.
		\end{equation*}
		Then (\ref{1}) gives that
		\begin{equation*}
			\left\| \delta_i\delta_j \tilde{W}_{\theta,n} \right\|_{L^1(\mathbb{T}_\theta^2)}=\left\| \tilde{\delta}_i\tilde{\delta}_j W_{\tilde{\theta},n} \right\|_{L^1(\mathbb{T}_{\tilde{\theta}}^2)}, \ \text{for } i,j=1,2.
		\end{equation*}
		 By defining $Q=\tilde{W}_{\theta,N}$, the result is obtained by the first part of the proof.
	\end{proof}
	
	\begin{remark}
	Notice that in the commutative case, the desired polynomials $Q$ are the polynomials $W_n$. While in that case the growth of the degree of $Q$ is standard (i.e. $O(e^{n^2})$, as $n\to+\infty$, see Remark \ref{rem2}), in the non-commutative setting, as the proof reveals, the degree of the polynomials $Q$ is of order $M(\theta,N)e^{O(N^2)}$, for $N\to+\infty$. Since $M(\theta,N)$ may depend on $N$ in an uncontrolled way, no degree-based quantitative estimate follows from this argument.
	\end{remark}
	
	\section{Concluding remarks}\label{Concl}
	
	First notice that, as we have already mentioned in the introduction, Ornstein's result from \cite{Ornstein} is Theorem 1, which is stated for homogeneous differential operators. Because derivations constitute the analogue of differentiation in rotation algebras, a generalisation of Ornstein's original statement in the non-commutative setting is still possible. Nonetheless, proving such an analogue in its full generality is not immediate or obvious. The reason is that if someone wants to follow Ornstein's original proof, then the localisation techniques that Ornstein uses would create some certain obstacles. Ornstein's theorem in its full generality for the non-commutative torus is still unknown to the authors.
	
	Our approach, motivated by \cite{KazWojc}, gives a counterexample of the inequality only for the case of trigonometric polynomials constructed by finite non-commutative Riesz products. The finiteness of the products is crucial and it actually allows us to extend the results from \cite{KazWojc} to the non-commutative setting and thus, to obtain a non-commutative analogue of \cite[Theorem 2.3]{KazWojc} for the anisotropic case, which also extends Theorem \ref{main}.
	
	But first, for sake of completeness, we give briefly the definition of a $d$-dimensional finite Riesz product on a $d$-dimensional quantum torus. For any natural number $d\ge2$ and any collection of irrational numbers $\{\theta_{kl}\}_{1\leq k<l\leq d}$ from $[0,1]$, let $\Theta$ denote the $d\times d$-skew symmetric matrix
	\begin{equation*}
		\Theta=[\theta_{kl}]_{k,l=1}^d.
	\end{equation*}
	Let $\mathcal{A}_\Theta$ denote the universal C$^*$-algebra generated by unitary operators $\{U_\Theta(j)\}_{j=1}^d$ that satisfy the following commutation relations:
	\begin{equation}\label{9}
		U_\Theta(k)U_\Theta(l)=e^{2\pi i\theta_{kl}}U_\Theta(l) U_\Theta(k), \ \text{for } k,l=1,2,\dots,d.
	\end{equation}
	Then $\mathcal{A}_\Theta$ will be the $d$-dimensional non-commutative or quantum torus; often denoted by $\mathbb{T}_\Theta^d$. If $U_\Theta:=(U_\Theta(1),\dots,U_\Theta(d))$ and for any multi-index $\alpha=(\alpha_1,\dots,\alpha_d)\in\mathbb{Z}^d$, we denote
	\begin{equation*}
		U_\Theta^\alpha:=U_\Theta^{\alpha_1}(1)U_\Theta^{\alpha_2}(2)\dots U_\Theta^{\alpha_d}(d),
	\end{equation*}
	then $\mathcal{A}_\Theta$ is the norm closure of the space of trigonometric polynomials of the form
	\begin{equation*}
		p(U_\Theta)=\sum_{\{ \left|\alpha_j\right|\leq N: \ j=1,\dots,d \}} p_\alpha U_\Theta^\alpha, \ \text{for } N\in\mathbb{N}_0.
	\end{equation*}
	Notice that this simply generalises the case of $d=2$, that we dealt with so far. Similarly, we can define the trace $\tau$ on $\mathcal{A}_\Theta$ by
	\begin{equation*}
		\tau(p)=p_{\mathbf{0}}, \ \text{for } \mathbf{0}:=(0,\dots,0),
	\end{equation*}
	and ensue likewise to build all the expected analogies in terms of Fourier coefficients, $L^p$-norms, etc. Moreover, let the derivations $\{\delta_j\}_{j=1}^d$ be defined by
	\begin{equation*}
		\delta_j(U_\Theta^\alpha):=2\pi i\alpha_j U_\Theta^\alpha, \ \text{for any } \alpha\in\mathbb{Z}^d,
	\end{equation*}
	and for any multi-index $\alpha\in\mathbb{N}_0^d$, the derivation operators
	\begin{equation*}
		D^\alpha:=\delta_1^{\alpha_1}\delta_2^{\alpha_2}\dots\delta_d^{\alpha_d}.
	\end{equation*}
	
	For any natural number $n$, let $\{m_k\}_{k\in\mathbb{N}}$ be a sequence of natural numbers that satisfies the growth condition (\ref{growth}). Let $\{\lambda_j\}_{j=1}^d$ be a collection of natural numbers and $(\xi_1,\dots,\xi_d)$ be a subset of $\{0,1\}^d$. Then, for any $j=1,\dots,d$, define the sequences $\{m_k^j\}_{k\in\mathbb{N}}$ by
	\begin{equation}\label{10}
		m_k^j=(-1)^{\xi_j k}(m_k)^{\lambda_j}, \ \text{for all } k\in\mathbb{N}.
	\end{equation}
	Let
	\begin{equation}\label{11}
		\mathbf{m}_k=(m_k^1,\dots,m_k^d), \ \text{for any } k\in\mathbb{N},
	\end{equation}
	and define the non-commutative Riesz products of $\mathbb{T}_{\Theta}^d$
	\begin{equation*}
		P_{\Theta,n}:=\left[ I+\frac{1}{2}\left( U_\Theta^{\mathbf{m}_n}+(U_\Theta^{\mathbf{m}_n})^{-1} \right) \right]\dots\left[ I+\frac{1}{2}\left( U_\Theta^{\mathbf{m}_1}+(U_\Theta^{\mathbf{m}_1})^{-1} \right) \right], \ \text{for any } n\in\mathbb{N}.
	\end{equation*}
	Then this construction, as in the 2-dimensional case, leads to the following theorem.
	
	\begin{thm}\label{thm0}
		Let $\alpha_1,\alpha_2,\dots,\alpha_m,\beta$ in $\mathbb{N}_0^d$ (where $d\geq2$) be an arbitrary choice of multi-indices that satisfy the following two conditions
		\begin{equation}\label{12}
			\left<\alpha_j,\Lambda\right>=\left<\beta,\Lambda\right>, \ \text{for all } j=1,2,\dots,m,
		\end{equation}
		for some $\Lambda=(\lambda_1,\dots,\lambda_d)\in\mathbb{N}^d$, and
		\begin{equation*}
			\left<\beta,\xi\right>\not\equiv\left<\alpha_1,\xi\right>\bmod2 \ \ \text{and } \ \left<\alpha_j,\xi\right>\equiv\left<\alpha_1,\xi\right>\bmod2, \ \text{for } j=2,\dots,m,
		\end{equation*}
		where $\xi=(\xi_1,\dots,\xi_d)\in\{0,1\}^d$. Then there exist positive constants $K$ and $C$ such that, for any choice of irrational numbers $\{\theta_{kl}\}_{1\leq k<l\leq d}$ of $[0,1]$, such that $1\cup\{\theta_{kl}\}_{1\leq k<l\leq d}$ is a $\mathbb{Q}$-linearly independent set, and any natural number $N$, there exists a trigonometric polynomial $Q$ of $\mathbb{T}_\Theta^d$, that depends on $N$, such that
		\begin{equation*}
			\sum_{j=1}^m \left\| D^{\alpha_j}Q \right\|_{L^1(\mathbb{T}_\Theta^d)}\leq K \ \ \text{and } \ \left\| D^\beta Q \right\|_{L^1(\mathbb{T}_\Theta^d)}\geq CN.
		\end{equation*}
	\end{thm}
	
	\noindent Since the proof follows the logic of the proof of Theorem \ref{main}, for the reader's convenience, we only give a sketch of it.
	
	\begin{proof}
	For any natural number $n$ and any $\left\{\theta_{kl}\right\}_{1\leq k<l\leq d}\subseteq[0,1]$, define the polynomials
	\begin{equation}\label{WN}
		W_{\Theta,n}:=\left(2\pi i\right)^{-\left|\alpha_m\right|}\sum_{j=1}^n\sum_{k\in M_j\cup(-M_j)}k^{-\alpha_m}  \hat{P}_{\Theta,n}(k)U_\Theta^{k},
	\end{equation}
	where
	\begin{equation*}
		M_j=\left\{ k\in\mathbb{Z}^d: \ k=\mathbf{m}_j+\sum_{l=1}^{j-1}\varepsilon_l\mathbf{m}_l, \ \varepsilon_l\in\{0,\pm1\} \right\}.
	\end{equation*}
	Notice that
	\begin{equation}\label{7}
		D^{\alpha_m}W_{\Theta,n}=P_{\Theta,n}-1
	\end{equation}
	and, for any $\alpha\in\{\alpha_1,\dots,\alpha_{m-1},\beta\}$,
	\begin{equation}\label{8}
		D^{\alpha}W_{\Theta,n}=\left(2\pi i\right)^{\left|\alpha\right|-\left|\alpha_m\right|}\sum_{j=1}^n\sum_{k\in M_j\cup(-M_j)}k^{\alpha-\alpha_m} \hat{P}_{\Theta,n}(k)U_\Theta^{k}.
	\end{equation}
	If $\Theta$ is the zero matrix, \cite[Theorem 2.3]{KazWojc} gives that there exist positive constants $K_1$ and $C_1$ such that, for any natural number $N$,
	\begin{equation*}
		\sum_{j=1}^m \left\| D^{\alpha_j}W_N \right\|_{L^1(\mathbb{T}^d)}\leq K_1 \ \ \text{and } \ \left\| D^\beta W_N \right\|_{L^1(\mathbb{T}^d)}\geq C_1N.
	\end{equation*}
	By the commutation relation (\ref{9}), as well as (\ref{7}) and (\ref{8}), for any $\alpha\in\{\alpha_j\}_{j=1}^m\cup\{\beta\}$, the polynomials $D^\alpha W_{\Theta,n}$ have Fourier coefficients that depend continuously on $\Theta$. By a multidimensional analogue of Lemma \ref{lem0}, we obtain some positive constants $K$ and $C$, such that for any natural number $N$, there exists some $\theta_0\in(0,1)$ (that depends on the constants and $N$) such that, for any skew symmetric matrix $\Theta=[\theta_{kl}]_{k,l=1}^d$, where $\theta_{kl}\in(-\theta_0,\theta_0)$,
	\begin{equation*}
		\sum_{j=1}^m \left\| D^{\alpha_j}W_{\Theta,N} \right\|_{L^1(\mathbb{T}_\Theta^d)}\leq K \ \ \text{and } \ \left\| D^\beta W_{\Theta,N} \right\|_{L^1(\mathbb{T}_\Theta^d)}\geq CN.
	\end{equation*}
	
	The interesting part is when there exists some irrational entry of $\Theta$ that does not belong to $(-\theta_0,\theta_0)$. As we did in Theorem \ref{main}, we look for a suitable natural number $M$ such that the new sequence $m'_j=Mm_j$ reduces us to the first step. By noticing that then the relation (\ref{10}) and (\ref{11}) become, respectively,
	\begin{equation*}
		{m'}_k^j=(-1)^{\xi_jk}M^{\lambda_j}(m_k)^{\lambda_j}, \ \text{for } k\in\mathbb{N} \ \text{and } j=1,\dots,d,
	\end{equation*}
	and
	\begin{equation*}
		{\mathbf{m}'}_k=({m'}_k^1,\dots,{m'}_k^d),\ \text{for } k\in\mathbb{N},
	\end{equation*}
	we observe that, for $1\leq j<l\leq d$,
	\begin{equation*}
		\tilde{U}_j^{m_k^{\lambda_j}} \tilde{U}_l^{m_k^{\lambda_l}}=e^{2\pi i\tilde{\theta}_{jl}m_k^{\lambda_l+\lambda_j}}\tilde{U}_l^{m_k^{\lambda_l}}\tilde{U}_j^{m_k^{\lambda_j}},
	\end{equation*}
	where
	\begin{equation*}
		\tilde{U}_s=U_\Theta^{M^{\lambda_s}}(s), \ \text{for } \ s=j,l \ \ \text{and } \ \tilde{\theta}_{jl}\equiv M^{\lambda_j+\lambda_l}\theta_{jl}\bmod\mathbb{Z}.
	\end{equation*}
	So we wish to find a natural number $M$ such that $\left\{\tilde{\theta}_{jl}\right\}_{1\leq j<l\leq d}\subseteq(0,\theta_0)$. This is possible if the multidimensional sequence
	\begin{equation*}
		\left( M^{\lambda_1+\lambda_2}\theta_{12}, M^{\lambda_1+\lambda_3}\theta_{13},\dots,M^{\lambda_1+\lambda_d}\theta_{1d},M^{\lambda_2+\lambda_3}\theta_{23},\dots,M^{\lambda_{d-1}+\lambda_d}\theta_{d-1,d} \right)
	\end{equation*}
	is uniformly distributed $\bmod\mathbb{Z}$ in $\mathbb{R}^r$, for $r=\tfrac{d(d-1)}{2}$. The latter is guaranteed by the $\mathbb{Q}$-linear independence condition and \cite[Theorem 1.6.4]{KuiNie}. Then the rest of the proof follows the steps of Theorem \ref{main}. In fact, if $\tilde{\Theta}=[\tilde{\theta}_{kl}]_{k,l=1}^d$ is the new skew symmetric matrix obtained from $\left\{\tilde{\theta}_{kl}\right\}_{kl=1}^d\subseteq(-\theta_0,\theta_0)$, $\tilde{\mathcal{A}}$ is the C$^*$-algebra generated by the unitaries $\{\tilde{U}_j\}_{j=1}^d$ and $\rho:\mathbb{T}_{\tilde{\Theta}}^d\to\tilde{\mathcal{A}}$ is the C$^*$-isomorphism such that $U_{\tilde{\Theta}}(j)\mapsto\tilde{U}_j$, for $j=1,\dots,d$, then the desired polynomial $Q$ will be given by
	\begin{equation*}
		Q=M^{-\left<\alpha,\Lambda\right>}\rho(W_{\tilde{\Theta},N}),
	\end{equation*}
	where $W_{\tilde{\Theta},N}$ is defined in (\ref{WN}) and $\alpha$ can be any of $\alpha_1,\dots,\alpha_m,\beta$ because of (\ref{12}).
	\end{proof}

	\begin{remark}
		Notice that Theorem \ref{main} is a special case of Theorem \ref{thm0} for $\alpha_1=(2,0)$, $\alpha_2=(0,2)$, $\beta=\Lambda=(1,1)$ and $\xi=(1,0)$.
	\end{remark}
	
	We close this section by noticing that there exists a stronger version of Lemma \ref{lem0} in the sense of Proposition \ref{prop0}. We clarify that this version is not necessary for proving Theorem \ref{main}, but it might be of independent interest. The proposition is motivated by the following aspects of continuity. The collection of the rotation algebras $\left\{\mathcal{A}_\theta\right\}_{\theta\in[0,1]}$ forms a continuous filed, in the sense that there exists a C$^*$-algebra $\mathcal{A}$ and surjective C$^*$-homomorphisms $\pi_\theta:\mathcal{A}\to\mathcal{A}_\theta$ such that the mappings $\theta\mapsto\left\| \pi_\theta(a) \right\|$ are continuous for any $a\in\mathcal{A}$ (cf. \cite{Elliott}). Therefore, the rotation algebras form a continuous field. Haagerup and R\o{}rdam prove in \cite{HaagRor} a stronger version of continuity. They show that for any infinite-dimensional separable Hilbert space $\mathcal{H}$, there exist a pair of continuous functions $U,V:[0,1]\to\mathbf{B}(\mathcal{H})$ (where $\mathbf{B}(\mathcal{H})$ is the set of bounded operators on $\mathcal{H}$), such that, for any $\theta\in[0,1]$, $U(\theta)$ and $V(\theta)$ are unitaries and canonical generators of the rotation algebra $\mathcal{A}_\theta$. Motivated by this, we prove the Proposition \ref{prop0}. Nonetheless, notice that the proposition describes a more general situation, in the sense that it applies to any choice of a common representing Hilbert space and generator functions $U$ and $V$, even if the latter are not continuous on $[0,1]$.
	
	\begin{prop}\label{prop0}
	Let $\mathcal{H}$ be an infinite-dimensional separable Hilbert space and $U$ and $V$ are functions on $[0,1]$, taking values in $\mathbf{B}(\mathcal{H})$ such that $U(\theta)$ and $V(\theta)$ are the canonical generators of the rotation algebra $\mathcal{A}_\theta$. If $\mathcal{A}$ is the following C$^*$-algebra
	\begin{equation}\label{13}
		\mathcal{A}=C^*\left( C([0,1])I,U,V \right)\subseteq\ell^\infty([0,1],\mathbf{B}(\mathcal{H})),
	\end{equation}
	then, for any $a\in\mathcal{A}$, the following mapping is continuous:
	\begin{equation*}
		\theta\mapsto\tau_\theta(a(\theta)),
	\end{equation*}
	where $\tau_\theta$ is the canonical trace on $\mathcal{A}_\theta$ and $a(\theta)$ is the value of the projection map $\pi_\theta:\mathcal{A}\to\mathcal{A}_\theta$ on $a$.
	\end{prop}
	
	\begin{proof}
	Notice that the algebra $\mathcal{A}$ in (\ref{13}) is the norm closure of the polynomials in the space $\ell^\infty([0,1],\mathbf{B}(\mathcal{H}))$. So,
	\begin{equation*}
		\left\| a \right\|=\sup_{\theta\in[0,1]}\left\| a(\theta) \right\|, \ \text{for any } a\in\mathcal{A}.
	\end{equation*}
	We first prove the result for any polynomial $P$ of $\mathcal{A}$ of the form
	\begin{equation*}
		P(\theta)=\sum_{(m,n)\in F}p_{m,n}(\theta)U^m(\theta)V^n(\theta),
	\end{equation*}
	where $F$ is a finite subset of $\mathbb{Z}^2$. Then the result is obtained due to density, since, for any $a,b\in\mathcal{A}$,
	\begin{equation*}
		\left| \tau_\theta(a(\theta))-\tau_\theta(b(\theta)) \right|\leq\left\| a(\theta)-b(\theta) \right\|\leq\left\| a-b \right\|, \ \text{for any } \theta\in[0,1].
	\end{equation*}
	For polynomials the result is obvious by simply noticing that
	\begin{equation*}
		\tau_\theta(P(\theta))=p_{0,0}(\theta)\in C([0,1]).
	\end{equation*}
	\end{proof}
	
	\begin{remark}
	Notice that in the Haagerup-R\o{}rdam case, the algebra $\mathcal{A}$, defined in (\ref{13}), is the continuous field of rotation algebras, which is a C$^*$-subalgebra of the space of continuous functions on $[0,1]$ taking values in $\mathbf{B}(\mathcal{H})$.
	\end{remark}
	
	\noindent As an immediate corollary to Proposition \ref{prop0}, we state the following.
	
	\begin{cor}
	Let $\mathcal{H}$ be an infinite-dimensional separable Hilbert space and $U$ and $V$ are functions on $[0,1]$, taking values in $\mathbf{B}(\mathcal{H})$ such that $U(\theta)$ and $V(\theta)$ are the canonical generators of the rotation algebra $\mathcal{A}_\theta$. If $\mathcal{A}$ is the C$^*$-algebra described in (\ref{13}), then the following mapping is continuous, for any $p>0$:
	\begin{equation*}
		\theta\mapsto\left\| a(\theta) \right\|_{L^p(\mathbb{T}_\theta^2)}, \ \text{for any } a\in\mathcal{A}.
	\end{equation*}
	\end{cor}
	
	\begin{proof}
	Notice that
	\begin{equation*}
		\left\| a(\theta) \right\|_{L^p(\mathbb{T}_\theta^2)}=\left( \tau_\theta(\left| a(\theta) \right|^p) \right)^{1/p}, \ \text{for any } a\in\mathcal{A}.
	\end{equation*}
	By continuous functional calculus, $\left| a \right|^p$ belongs to $\mathcal{A}$ and moreover,
	\begin{equation*}
		\left| a \right|^p(\theta)=\left| a(\theta) \right|^p\in\mathcal{A}_\theta.
	\end{equation*}
	Then the result is obtained by applying Proposition \ref{prop0}.
	\end{proof}

	\bibliographystyle{plain}
	\bibliography{notes-bibl_rev}
	
\end{document}